\documentclass[a4paper]{article}

\usepackage{fullpage,url,amsmath,amsfonts,amssymb, tedmath}
\newcommand{\tr}{\mbox{tr}}
\newcommand{\cc}{\mathbb{C}}

\newcommand{\cir}{\,\text{circ}}

\title{Representations of group rings and groups}
\author{ 
Ted
 Hurley\footnote{National Universiy of Ireland Galway, email:
 Ted.Hurley@NuiGalway.ie }}
\date{}

\begin{document}

\maketitle



\begin{abstract}   An isomorphism between the group ring of a finite
  group and a ring of certain block diagonal matrices is established.  
The group ring $RG$ of a finite group $G$ is 
isomorphic to the set of {\em group ring matrices} over $R$. 
It is shown  that for any group
ring matrix $A$ of $\cc G$ there exists a
matrix $P$ (independent of the entries of $A$) such that $P^{-1}AP= \diag(T_1,T_2,\ldots,
T_r)$ for block matrices $T_i$ of fixed size $s_i\ti s_i$ where $r$ is the
number of conjugacy classes of $G$ and $s_i$ are the ranks of the
group ring matrices of the primitive idempotents. 
Using the isomorphism of the group ring to the ring of group ring
matrices followed by the mapping $A\mapsto P^{-1}AP$  (where $P$ is of
course fixed) gives an isomorphism from the group ring to the ring of
such block matrices. Specialising to the group elements gives a faithful 
representation of the group. Other 
representations of $G$ may be derived using the blocks in  the  
images of the group elements. 

Examples are given demonstrating how interesting and useful 
representations of groups 
can be derived using  the method. 


For a  finite abelian group $Q$
 an explicit  matrix $P$ is given which diagonalises any  group 
ring matrix of $\cc Q$. The matrix $P$ is defined directly in terms of
roots of unity depending only on an expression for $Q$ as a product of
cyclic groups.   
The characters and character table of $Q$
 may be read off directly from the rows of the diagonalising 
matrix $P$. 
This  has applications 
to signal processing and generalises the cyclic group case.   



\end{abstract}
\section{Introduction} 

For background on groups and 
group rings, including information on conjugacy classes and 
representation theory, 
 see \cite{seh},  and for 
group ring matrices see \cite{hur3}. Further information on
 representation theory and character theory may be found in 
\cite{curtis} and/or \cite{d2n}.  
Results are given over 
the complex numbers $\cc$ but many of the 
results  hold over other suitably chosen fields.
 
A matrix $A$ is said to be {\em diagonalised by $P$} if $P^{-1}AP=D$ where
$D$ is a diagonal matrix. 
A circulant matrix can be 
diagonalised by the Fourier matrix of the same size. 
The diagonalising Fourier 
matrix is independent of the particular circulant 
matrix; this is the basis
for the finite Fourier transform and the convolution theorem, 
see for example \cite{blahut}.
The Fourier $n\ti n$ matrix satisfies $FF^*=nI_n$, (and is thus a complex
Hadamard matrix)  and when
the rows are labelled by $\{1,g,g^2,\ldots,g^{n-1}\}$, it gives the
characters and character table of the cyclic group $C_n$ generated by $g$.
The ring of  circulant matrices over $R$ is
isomorphic to the ring of {\it group ring matrices} over $R$ of the
cyclic group, see for example \cite{hur3}. 

The group ring
of a finite group is isomorphic to the ring  of group ring
matrices as determined in \cite{hur3}. 
The group ring matrices are types of matrices determined by their
first rows; 
 see section \ref{sec1} below for precise formulation.
For example circulant matrices are the group ring matrices of the cyclic
group and matrices of the form  $\begin{pmatrix} A&B
		       \\ B\T&A\T\end{pmatrix}$, where $A,B$ are
                       circulant matrices, are  
determined by their first rows and correspond to the group 
ring matrices  of the dihedral group. See Sections \ref{sec1},\ref{cases} and
\ref{sec2} below for further examples.  

Group rings and group ring matrices will be over $\cc$ unless
otherwise stated. Results may hold over other fields but these are not
dealt with here. 

An isomorphism
from the ring of group ring matrices of a finite group $G$ into certain block
diagonal matrices is established. 
More precisely it is shown that for a group ring matrix $A$ 
 of a finite group $G$ there exists a
matrix $P$ (independent of the particular $A$) 
such that $P^{-1}AP= \diag(T_1,T_2,\ldots,
T_r)$ for block matrices $T_i$ of fixed size $s_i\ti s_i$ where $r$ is the
number of conjugacy classes of $G$ and the $s_i$ are the ranks of the
group ring matrices of the primitive idempotents. Thus the group ring
$\cc G$  is isomorphic to 
matrices of the type $\diag (T_1,T_2,\ldots, T_r)$. A faithful representation
of the group itself may be given by taking images of the group elements.
 Other representations of 
 $G$ may be obtained using the blocks in the images of the of the
 group elements. 

See Sections \ref{cases} and \ref{sec2} below for applications
and examples; these show how interesting and useful representations of the
groups, and group rings,  may be derived by the method.

The finite abelian group ring  
is a special case  
but is dealt with
independently in Section \ref{sec2} as more direct information and
direct calculations may be made. The diagonalising matrix is 
obtained directly from Fourier type matrices, 
the diagonal entries are obtained  
from the entries of
the first row of the group ring matrix  and the
character table may be read off from the diagonalising matrix.  

More precisely, for a
given finite abelian group $H$ it is shown explicitly 
that  exists a matrix $P$ such that
$P^{-1}BP$ is diagonal where $B$ is any group ring matrix of $H$. The
matrix $P$ is independent of the entries of the particular group ring
matrix $B$ and  the diagonal entries are given precisely in terms
of the entries of the first row of $B$. The matrix $P$ may be chosen
so that $PP^*=nI_n$ and
when the rows of $P$ are labelled 
appropriately to the structure of the group as a product of cyclic
groups, then the rows of $P$ give the characters and character table
of $H$.  


Many results for circulant $n\ti n $ 
matrices (= group ring matrices over the cyclic group $C_n$) 
hold not just over $\cc$ but over any field $F$ which contains a
primitive  $n^{th}$ root of unity.  
Similarly some  results here    
hold over  fields other than $\cc$ but this aspect is not dealt with here.

The idea of using group ring matrices of 
complete orthogonal sets of idempotents originated
in \cite{hur5} where these are used in the study and construction of 
 types of multidimensional paraunitary matrices.

\section{Group ring matrices}\label{sec1}

Certain classes of 
matrices are determined by their first row or
column. A particular type of such matrices are those corresponding to
group rings. It is shown in \cite{hur3} that the group ring $RG$ where
$|G|=n$ may be embedded in the ring of ${n\ti n}$ matrices over $R$ in
a precise manner. 

Let $\{g_1, g_2, \ldots , g_n\}$ be a fixed listing of the elements of $G$.
Consider the following matrix:

$\left(\begin{array}{lllll}
{g_1^{-1}g_1} & {g_1^{-1}g_2} & {g_1^{-1}g_3} 
 &  \ldots & {g_1^{-1}g_n} \\

{g_2^{-1}g_1} & {g_2^{-1}g_2} & {g_2^{-1}g_3} 
 &  \ldots & {g_2^{-1}g_n} \\ 

\vdots & \vdots & \vdots &\vdots &\vdots \\

{g_n^{-1}g_1} & {g_n^{-1}g_2} &{g_n^{-1}g_3} 
 &  \ldots & {g_n^{-1}g_n} 
\end{array}\right)$

Call this the {\em matrix of $G$} (relative to this listing)  and denote
it by $M(G)$. 
\subsection{$RG$-matrix} 

 Given a listing of the elements of $G$, form the
matrix $M(G)$ 
 of $G$ relative to this listing. 
An {\em  $RG$-matrix} over a ring $R$ is a matrix obtained by substituting elements of $R$ for 
 the elements of $G$ in $M(G)$. If $w\in RG$ and $w=\di\sum_{i=1}^n
\al_ig_i$ then $\sigma(w)$ is the $n\ti n$ $RG$-matrix obtained by
substituting each $\al_i$ for $g_i$ in the group matrix.
 
Precisely $\sigma(w)= \left(\begin{array}{lllll}
\alpha_{g_1^{-1}g_1} & \alpha_{g_1^{-1}g_2} &\alpha_{g_1^{-1}g_3} 
 &  \ldots & \alpha_{g_1^{-1}g_n} \\

\alpha_{g_2^{-1}g_1} & \alpha_{g_2^{-1}g_2} &\alpha_{g_2^{-1}g_3} 
 &  \ldots & \alpha_{g_2^{-1}g_n} \\ 

\vdots & \vdots & \vdots &\vdots &\vdots \\

\alpha_{g_n^{-1}g_1} & \alpha_{g_n^{-1}g_2} &\alpha_{g_n^{-1}g_3} 
 &  \ldots & \alpha_{g_n^{-1}g_n} 
\end{array}\right) $

It is shown in \cite{hur3} that $w\mapsto \sigma(w)$ gives an isomorphism
of the group ring $RG$ into the ring of $n\ti n$ matrices over $R$. 

Given the entries of the first row of an  $RG$-matrix 
the entries of the other rows are determined from the matrix $M(G)$ of
$G$. 

An $RG$-matrix is a matrix corresponding
to a group ring element in the isomorphism from the group ring
into the ring of $R_{n\times n}$  matrices. The
isomorphism depends on the listing of the elements of $G$. For example 
if $G$ is cyclic, an $RG$-matrix is a circulant matrix 
relevant to the natural listing $G=\{1,g,g^2,\ldots, g^{n-1}\}$ where $G$ is
generated by $g$. An $RG$-matrix when $G$ is dihedral is one of the form
  $\begin{pmatrix} A&B \\ B&A\end{pmatrix}$ where $A$ is circulant and
  $B$ is reverse circulant but also one of the form $\begin{pmatrix} A&B
						     \\
						     B\T&A\T\end{pmatrix}$
						     where both $A,B$
						     are circulant, in
  a different listing of $G$.  
Other examples are given within \cite{hur3}.

In general given a group ring element $w$, and a fixed  listing of the
elements of the group,  
the corresponding capital letter $W$ is often used to denote the image of  $w$,
$\sigma(w)$, in the ring of $RG$-matrices.

\subsubsection*{Listing}
Changing the listing of the elements of the group  gives an equivalent $RG$-matrix
and  one is obtained from the
other by a sequence of processes consisting of interchanging two rows and then
interchanging the corresponding two columns.

\section{Block diagonal}\label{block}
Matrices, when diagonalisable, may be simultaneously diagonalised if
and only if they commute. 
However a set of matrices may be simultaneously block diagonalisable
in the sense that there exist a matrix $U$ such that $U^{-1}AU$ has
the form $\diag(T_1,T_2,\ldots, T_r)$, where each $T_i$ is of fixed
$r_i\ti r_i$ size, for every matrix $A$ in the set -- and for $U$
independent of $A$. This is the case for group ring matrices.

Idempotents will naturally play an important part. (See \cite{hur5}
where these are used for paraunitary matrices.) Say $e$ is an
idempotent in a ring $R$ if $e^2=e$ and say $\{e,f\}$ are orthogonal if
$ef=0=fe$. Say $\{e_1,e_2, \ldots, e_k\}$ is a complete orthogonal set of
idempotents in a ring $R$ if $e_i^2=e_i, e_ie_j= 0$ for $i\neq j$ and
$e_1+e_2+\ldots+ e_k =1$ where $1 $ is the identity of $R$. Now $\tr
A$ denotes the trace of a matrix $A$.  

\begin{proposition}\label{gr1} Suppose $\{e_1,e_2,\ldots,e_k\}$ is a
  complete orthogonal set of idempotents. Consider
 $w= \al_1 e_1 + \al_2 e_2 + \ldots +
\al_ke_k$ with $\al_i \in F$, a field. Then $w$
 is invertible if and only if each $\al_i \neq 0$ 
and in this case $w^{-1}
 = {\al_1}^{-1}e_1 + {\al_2}^{-1}e_2+ \ldots + {\al_k}^{-1}e_k$.
\end{proposition} 

\begin{proof} Suppose each $\al_i \neq 0$.
Then $w*({\al_1}^{-1}e_1+{\al_2}^{-1}e_2+ \ldots + {\al_k}^{-1}e_k)
 = e_1^2 + e_2^2 + \ldots + e_k^2 = e_1+e_2+\ldots + e_k = 1$.

Suppose $w$ is invertible and that some $\al_i=0$. 
Then $we_i =0$ and so $w$ is a (non-zero) zero-divisor and is not invertible.
\end{proof}

\begin{lemma}\label{trrank} Let $\{E_1,E_2, \ldots, E_s\}$ be a
set of orthogonal idempotent matrices. Then $\rank
(E_1+E_2 +\ldots + E_s) = \tr (E_1+E_2+ \ldots + E_s) = \tr E_1+ \tr
 E_2+ \ldots + \tr E_s = \rank E_1+ \rank E_2 + \ldots +\rank
E_s$.
\end{lemma}
\begin{proof}
It is known that $\rank A = \tr A$ for an idempotent matrix, see
for example \cite{idemrank}, and so
$\rank E_i = \tr E_i$ for each $i$. If $\{E,F,G\}$ is a set of orthogonal
 idempotent matrices so is  $\{E+F,G\}$. From this it follows that $\rank
(E_1+E_2 +\ldots + E_s) = \tr (E_1+E_2+ \ldots E_s)= \tr E_1+\tr E_2 +
 \ldots + \tr E_s = \rank E_1+ \rank E_2 + \ldots +\rank
E_s$.
\end{proof}
\begin{corollary}\label{trrank1}
$\rank(E_{i_1}+ E_{i_2}+ \ldots + E_{i_k})= 
\rank E_{i_1} +\rank E_{i_2}+ \ldots + \rank E_{i_k}$ for $i_j \in \{
1,2,\ldots, s\}$, $i_j\neq i_l$ for $j\neq l$.
\end{corollary}

Let $A= a_1 E_1 + a_2 E_2 + \ldots + a_kE_k$ for a complete set of 
idempotent orthogonal
 matrices $E_i$. Then $A$
 is invertible if and only if each $a_i \neq 0$ and in this case $A^{-1}
 = a_1^{-1}E_1 + {a_2}^{-1}E_2+ \ldots + {a_k}^{-1}E_k$. 
This is a special case of the following.  

\begin{proposition}{\label{det2}} Suppose $\{E_1, E_2, \ldots, E_k\}$ is a
 complete symmetric orthogonal set of idempotents in $F_{n\ti n}$. 
Let $Q= a_1 E_1 + a_2 E_2 + \ldots +
 a_kE_k$.  Then the determinant of $Q$ is 
$|Q| = a_1^{\rank E_1}a_2^{\rank E_2}\ldots a_k^{\rank E_k}$.
\end{proposition}

Let $RG$ be the group ring of a finite group $G$ over the ring $R$.
Let $\{e_1,e_2, \ldots, e_k\}$ be a
complete orthogonal set of idempotents in $RG$ and $\{E_1,E_2,
\ldots,E_k\}$ the corresponding $RG$-matrices (relevant to some listing
of the elements of $G$). Such a set of idempotents is known to exist
when $R=\cc$, the complex numbers,  and also over other fields, see
for example \cite{curtis} or \cite{seh}. We
will confine ourselves here to $\cc$ but many of the results hold over
these other fields. 
The idempotent elements from the group ring satisfy $e^*=e$  and so the
idempotent matrices  
are symmetric, $E^*=E$, and satisfy $E^2=EE^*=E$.



We now specialise the $E_i$ to be $n\ti n$ matrices corresponding to
the group ring idempotents $e_i $, that is $\sigma e_i=E_i$. 
Define the rank of $e_i$ to be that of $E_i$.

Consider now the group ring $FG$ where $F=\cc$ the complex numbers and
$G$ is a finite group. As already mentioned, $FG$ contains a complete
orthogonal set of idempotents $\{e_1,e_2,\ldots, e_k\}$ which may be
taken to be primitive, \cite{seh}.  
\begin{theorem}\label{oneone} Let $A$ be a $FG$-matrix with  
				 $F=\cc $. 
Then there exists a
 non-singular matrix $P$ independent of $A$ such that $P^{-1} AP = T$
where $T$ is a block diagonal matrix with blocks of size $r_i \times
r_i$ for $i=1,2, \ldots, k$ and $r_i$ are the 
ranks of the $e_i$.
\end{theorem}
\begin{proof} Let $\{e_1,e_2,\ldots, e_k\}$ be the orthogonal idempotents
  and $S=\{E_1, E_2, \ldots, E_k\}$ the group ring matrices corresponding to
  these, that is, $\sigma(e_i)=E_i$ in the embedding of the group ring
 into the $\cc G$-matrices. 
Any column of $E_i$ is orthogonal to any column of $E_j$ for $i\neq
j$ as $E_iE_j^*=0$. Now let $\rank E_i= r_i$. Then $\sum_{i=1}^k r_i= n$. 
Let $S_i=\{v_{i,1}, v_{i,2}, \ldots v_{i,r_i}\}$ be a basis for the column
space of $E_i$ consisting of a subset of the columns of
 $E_i$; do this for each $i$. Then each element of $S_i$ is orthogonal to each
element of $S_j$ for $i\neq j$. Since  $\sum_{i=1}^k r_i= n$ it follows that
$S=\{S_1, S_2, \ldots, S_k\}$ is a basis for $F^n$.  

Let  $V_{i,j} $ denote the $FG$-matrix determined by the column
vector $v_{i,j}$, let $S_i(G)$ denote the set of $FG$-matrices
obtained by substituting $V_{i,j}$ for $v_{i,j}$ in $S_i$ and 
let $S(G)$ denote the set of $FG$-matrices
obtained by substituting  $S_i(G)$ for  $S_i$ in $S$. 

As $S$ is a basis for $F^n$ the first column of $AE_i$ is a linear
 combination of elements from $S$. 
The first column of $AE_i$ determines $AE_i$, as 
 $AE_i$ is an $FG$-matrix, and
 hence $AE_i$ is a linear combination of elements of $S(G)$. By
 multiplying $AE_i$ through on the right by $E_i$, and orthogonality, it follows that $AE_i$
 is a linear combination of $S_i(G)=\{V_{i,1}, V_{i,2}, \ldots,
 V_{i,r_i}\}$. Now each $V_{i,j}$ consists of columns which are a
 permutation of the columns of $E_i$. Also $E_i$ contains the columns
 $S_i$. Thus equating $AE_i$ to the linear combination of $S_i(G)$
 implies that each $AV_{i,j}$ is a linear combination of $S_i$. 

 Let $P$ then be the matrix with columns consisting of the
 first columns $S_i$ for $i=1,2, \ldots, k$. Then  
 $AP =PA$ where $T$ is a matrix of blocks of  size $r_i\ti r_i$ arranged
 diagonally for $ i=1,2, \ldots k$. Since $P$ is invertible it follows
 that $P^{-1}AP = T$.

\end{proof}
 
The proof is constructive in the sense that the matrix $P$ is
constructed from the complete orthogonal set of idempotents. Method:
\begin{enumerate}
\item Find  complete orthogonal set of idempotents $\{e_1,e_2,\ldots,
      e_k\}$ for $FG$.
\item Construct the corresponding $FG$-matrices
      $\{E_1,E_2,\ldots,E_k\}$.
\item Find a basis $S_i$ for the column space of  $E_i$ for $1\leq 1\leq
      k$.
\item  Let $P$ be the matrix made up of columns of the union of the
       $S_i$.
\item Then $P^{-1}AP$ is a block diagonal matrix consisting of blocks of
      size $r_i\ti r_i$ where $r_i$ is the rank of $E_i$. 
\end{enumerate}
However this algorithm 
requires being able to construct  a complete orthogonal set of
idempotents. If the
matrix $P$ could be obtained directly then indeed this would be a way for the
construction of the idempotents. 

\begin{corollary} The group ring $FG$ is isomorphic to a subring of
  such block diagonal matrices. The isomorphism is given by $w\mapsto
  \sigma(w) = W \mapsto P^{-1}WP$. 
\end{corollary} 
The isomorphism includes an  isomorphic embedding of the group $G$ itself
into the set of such block diagonal matrices.  Other linear
representations of $G$ may be obtained by  using the block images
of the group elements.

\begin{theorem}\label{onetwo}  Suppose $A$ is an $FG$-matrix where
  $F=\cc $. Then there  exists a 
 unitary matrix $P$ such that $P^T AP = T$ where $T$ is a block 
diagonal matrix with blocks of size $r_i \times r_i$ for $i=1,2, \ldots,
 k$ along the diagonal. 
\end{theorem}
\begin{proof}
The diagonalising matrix in the proof of Theorem \ref{oneone} 
may be made unitary by constructing an
orthonormal basis for space generated by $\{V_{i,1}, V_{i,2}, \ldots,
V_{i,r_i}\}$ for each $i=1,2,\ldots, k$. Let $S_i = \{W_{i,1},W_{i,2},
\ldots , W_{i,r_i}\}$ be an orthonormal basis for the space spanned by
$\{V_{i,1}, V_{i,2}, \ldots, V_{i,r_i}\}$. Then $\hat{S} = \{S_1, S_2,
\ldots, S_k\}$ is an orthonormal basis for $F^n$. Set $P$ to be the
matrix with elements of $\hat{S}$ as columns. Then $P$ is unitary and 
$P\T AP=T$ as required.  
\end{proof} 

The group ring is isomorphic to the ring of $RG$-matrices,
\cite{hur3}, and the ring of $RG$ matrices is isomorphic to the ring
of such block diagonal matrices under the mapping $w\mapsto \sigma(w)=W
\mapsto P^{-1}WP$ for this fixed $P$.  

\section{Cases, applications}\label{cases}
See for example
\cite{curtis,d2n,seh} for information on representation theory
including characters and character tables. See for example \cite{hur3}
for information on group ring matrices and in particular on the method
for obtaining the corresponding group ring matrix from a group ring
element. 

\subsection{Abelian} 
When $G=C_n$, the cyclic group of order $n$,  the matrix $P$ of Theorem
\ref{oneone} is the Fourier matrix and $T$ is a diagonal matrix. 
The case when $G$ is any abelian group is dealt with fully in section 
 \ref{sec2}.  

\subsection{Dihedral} The dihedral group $D_{2n}$ is generated by
elements $a$ and $b$ with presentation: 

$    \langle a,b \;|\; a^n = 1,\, b^2 = 1,\, bab = a^{-1} \rangle $

It has order $2n$, and a natural listing of the  elements is
 $\{1,a,a^2,\ldots,a^{n-1}, b,ab,a^2b,\ldots,a^{n-1}b\} $.

As every element
in $D_{2n}$ is conjugate to its inverse, the complex characters of $D_{2n}$ are
real.
The characters $D_{2n}$ 
are contained in an
extension of  $\mathbb{Q}$ of degree $\phi(n)/2$ and  this is 
$\mathbb{Q}$ only for $2n\leq 6$.   Here $\phi$ is the Euler phi function.
Let $S_n$ denote the symmetric group of order $n$. 
The characters of $S_n$ are rational. 

\subsubsection{$S_3=D_6$}
Consider $D_6$. Note that $D_6 =S_3$. 
The conjugacy classes are $\{1\}, \{a,a^2\},
  \{b,ab,ab^2\}$. The central (primitive, symmetric) idempotents are 
$e_0=1/6(1+a+a^2+b+ba+ba^2), e_1=1/6(1+a+a^2 - b-ba-ba^2),
  e_3=1/3(2-a-a^2)$.

This gives the corresponding group ring matrices:
$$E_0=\frac{1}{6} \left(\begin{smallmatrix} 1 & 1&1&1&1&1\\
  1&1&1&1&1&1 \\ 1&1&1&1&1&1 \\ 1&1&1&1&1&1 \\ 1&1&1&1&1&1
  \\ 1& 1&1&1&1&1 \end{smallmatrix}\right),
E_1=\frac{1}{6} \left(\begin{smallmatrix}1 & 1&1&-1&-1&-1\\
  1&1&1&-1&-1&-1 \\ 1&1&1&-1&-1&-1 \\ -1&-1&-1&1&1&1 \\ -1&-1&-1&1&1&1
  \\ -1& -1&-1&1&1&1 \end{smallmatrix}\right),
E_2=\frac{1}{3} \left(\begin{smallmatrix} 2 & -1&-1&0&0&0\\
  -1&2&-1&0&0&0 \\ -1&-1&2&0&0&0 \\ 0&0&0&2&-1&-1 \\ 0&0&0&-1&2&-1
  \\ 0& 0&0&-1&-1&2 \end{smallmatrix}\right).$$

Now $E_0,E_1$ have rank $1$ and $E_2$ has rank $4$ from general theory.

Thus we need a set consisting of 
one column from each of $E_0,E_1$ and 4 linearly
independent columns from $E_2$ to form a set of $6$ linearly
independent vectors. It is easy to see that

\noindent $v_1=(1,1,1,1,1,1)\T, v_2=(1,1,1,-1,-1,-1)\T, 
v_3=(2,-1,-1,0,0,0)\T, \\ v_4=(-1,2,-1,0,0,0)\T,
v_6=(0,0,0,2,-1,-1)\T,v_6=(0,0,0,-1,2,-1)\T$

is such a set. 

Now let  $P=(v_1,v_2,v_3,v_4,v_5,v_6)$.
Then for any $\cc D_6$ matrix $A$, $P^{-1}AP = \diag(a,b,D)$ where $D$ is
a $4\ti 4$ matrix. 


As noted, the group ring is isomorphic to the ring of $RG$-matrices,
and the ring of $RG$-matrices is isomorphic to the ring
of such block diagonal matrices under the mapping $A\mapsto P^{-1}AP$
for this fixed $P$ (Theorem \ref{oneone}). 

Now consider the image of $D_6$ itself under this isomorphism. 
The matrix $A$ of $a \in S_3=D_{6}$ in this isomorphism
is mapped to $P^{-1}AP$. 
Here $A=\left(\begin{array}{ccc|ccc}0&1&0 &0&0&0 \\ 0&0&1&0&0&0 \\ 1&0&0
&0&0&0 \\ \hline 0&0&0&0&0&1 \\ 0&0&0&1&0&0 \\ 0&0&0&0&1&0 \end{array}\right)$

and 
$P^{-1}AP=\left(\begin{array}{cc|cccc} 1 & 0& 0&0&0&0
  \\0&1&0&0&0&0 \\ \hline 0&0& -1 &1 &0&0\\0&0&-1&0&0&0 \\ 0&0&0&0&0&-1
  \\ 0&0&0&0&1&-1 \end{array}\right)$.

(In some cases it is easier to work out $AP$ and then solve for $D$
in $PD$ where $D$ is of the correct block diagonal type.)  

Similarly the image of $b$ is obtained; $B$ is the $RG$-matrix of $b$
and  
$P^{-1}BP = \left(\begin{array}{cc|cccc} 1
  &0&0&0&0&0 \\ 0&-1 
  &0&0&0&0 \\ \hline 0&0& 0&0& 1& 0 \\ 0&0&0&0&0&1 \\ 0&0&1&0&0&0
  \\ 0&0&0&1&0&0 \end{array}\right)$.

Representations of $S_3=D_{6}$ may be obtained using the block
matrices of the images of the group elements. For example
 $a\mapsto \begin{pmatrix} -1 &1 &0&0 \\ -1&0&0&0 \\ 0&0&0&-1
  \\ 0&0&1&-1\end{pmatrix}, b\mapsto \begin{pmatrix} 0&0&1&0
    \\ 0&0&0&1 \\ 1&0&0&0\\ 0&1&0&0 \end{pmatrix}$
gives a representation of $D_6=S_3$. 
 
It may be shown directly from the structure of $P$ and of $A$ 
corresponding to a group element $a$ that the $4\ti 4$ part
in $P^{-1}AP$ 
has the
form $T=\begin{pmatrix} X&0\\0&Y \end{pmatrix}$ or of the form
$S=\begin{pmatrix} 0&X\\Y&0 \end{pmatrix}$ where $X,Y$ are $2\ti 2$
blocks. 

Say a matrix is in $\mathbb{T}$ if it has the form $T$ and say a
matrix is in $\mathbb{S}$ if it has the form $S$. Interestingly then generally
$TS\in \mathbb{S}, ST \in \mathbb{S}, T_1T_2\in \mathbb{T},
S_1S_2 \in \mathbb{T}$, for $S,S_1,S_2 \in \mathbb{S}, T,
T_1,T_2 \in \mathbb{T}$.  

\subsubsection{Unitary required?}
Now $P$ may be made orthogonal by finding an orthogonal basis for the
$4$ linearly independent columns of $E_2$ and then dividing each of the
resulting set of 6 vectors by their lengths.

An orthogonal basis for the columns of $E_2$ is 

$\{(2,-1,-1,0,0,0)\T, (0,1,-1,0,0,0)\T, (0,0,0,2,-1,-1)\T,
(0,0,0,0,1,-1)\T\}$. 

Construct an orthonormal basis: 

\noindent $v_1=\sqrt{\frac{1}{6}}(1,1,1,1,1,1)\T,
v_2=\sqrt{\frac{1}{6}}(1,1,1,-1,-1,-1)\T,
v_3=\sqrt{\frac{1}{6}}(2,-1,-1,0,0,0)\T,
\\ v_4=\sqrt{\frac{1}{2}}(0,1,-1,0,0,0)\T,  
v_5=\sqrt{\frac{1}{6}}(0,0,0,2,-1,-1)\T,
v_6= \sqrt{\frac{1}{2}}(0,0,0,0,1,-1)\T$.

Now construct the unitary  (orthogonal in this case) 
 matrix $P=(v_1,v_2,v_3,v_4,v_5,v_6)$.
Then for any $\cc D_6$ matrix $A$, $P^*AP = \diag(a,b,D)$ where $D$ is
a $4\ti 4$ matrix. 

When $P$ is unitary, = orthogonal in this case, then $P\T AP$ and $P\T
BP$ are unitary as $A,B$ are orthogonal. 
The diagonal $4\times 4$ matrix must then be
orthogonal. For
example: 

$P\T AP= P^*AP= \left(\begin{smallmatrix} 1&0&0&0&0&0 \\ 0&1&0&0&0&0 \\ 0&0& -1/2
	 &\sqrt{3}/2 & 0&0 \\ 0&0&-\sqrt{3}/2 & -1/2 &0&0 \\ 0&0&0&0&-1/2
	 & -\sqrt{3}/2 \\ 0&0& 0&0& \sqrt{3}/2 &-1/2 \end{smallmatrix}\right)$

The $4\ti 4$ block matrix 
is easily checked to be unitary/orthogonal as expected   from
theory. 

\subsection{Other dihedral}{\label{symm6}} 

The character tables for $D_{2n}$ may be derived from 
\cite{curtis,d2n} and are also available at various on-line resources
such as that of Jim Belk.
We outline how the results may be applied in the case of
$D_{10}$. 

The character table of $D_{10}$ is the following:
$\begin{pmatrix}1 &b& a& a^2 \\ 1&5&2&2 \\ \hline 1&1&1&1 \\ 1&-1&1&1
  \\ 2&0&2\cos(2\pi/5)&2\cos(4\pi/5) \\ 2 &0 & 2\cos(4\pi/5) &
  2\cos(8\pi/5) \end{pmatrix}$.

This gives the following complete (symmetric) orthogonal set of
idempotents in the group ring:
$e_0 = \frac{1}{10}(1+a+a^2 + a^3+a^4 + b+ba+ba^2+ba^3+ba^4), 
e_1=\frac{1}{10}(1+a+a^2+a^3+a^4 - b
-ba-ba^2-ba^3-ba^4), e_2= \frac{4}{10}(1+\cos(2\pi/5)a +
\cos(4\pi/5)a^2+\cos(4\pi/5)a^3+\cos(2\pi/5)a^4),
e_3=\frac{4}{10}(1+\cos(4\pi/5)a +\cos(8\pi/5)a^2+\cos(8\pi/5)a^3+
\cos(4\pi/5)a^4)$.

Let $\sigma(e_i)=E_i$ -- this is the image of the group ring element
$e_i$ in the group ring matrix. 
Each of $E_0,E_1$ has rank $1$ and each of $E_2,E_3$ has rank
$4$.  Four linearly independent columns in each of $E_2,E_3$ are easy
to obtain and indeed  four orthogonal such may be derived if
required. The matrix $P$ is formed using the first columns of
$E_1,E_2$ and $4$ linearly independent columns of each of $E_3$ and $E_4$.  
Then $P^{-1}AP =
\diag(\al_1,\al_2, T_1, T_2)$ for any group ring matrix $A$ of
$D_{10}$ where $T_1,T_2$ are $4\ti 4$ block matrices. Then the
composition of mappings $w \mapsto \sigma(w)= W \mapsto P^{-1}WP$ is an
isomorphism. Representations
of  the group may be obtained by specialising to blocks of the 
images of the group elements.

The form of $P$ is $\begin{pmatrix} A & C &0 & D & 0
					 \\ B & 0 & C_1 & 0 & D_1 \end{pmatrix}$
for suitable $5\times 2$ blocks $A,0,C,D,C_1,D_1$.
Then it may be shown that in $P^{-1}AP$ the two $4\times 4$ blocks have
the form $\begin{pmatrix}X& 0 \\ 0& Y \end{pmatrix}$ or else the form
$\begin{pmatrix} 0 & X \\ Y& 0\end{pmatrix}$ for $2\times 2$ blocks
  $X,Y$ when $A$ corresponds to a group element $a$. 
\subsection{Quaternion group of order $8$}
The five primitive central idempotents $\{e_1,e_2,e_3,e_4,e_5\}$ of $\cc K_8$ where $K_8$ is the
quaternion group of order $8$ is given in \cite{seh} page 186. 
$K_8= \langle a,b \, |
\, a^4=1, a^2=b^2, bab^{-1}=a^{-1}\rangle$ and is listed as 
$\{1,a,a^2,a^3, b,ab, a^2b, a^3b\}$.

\begin{eqnarray*} e_1&= & 1/8(1+a+a^2+a^3+b+ab+a^2b+a^3b) \\
e_2&= &1/8(1+a+a^2+a^3 - b-ab-a^2b-a^3b) \\ e_3 &= &
1/8(1-a+a^2-a^3+b-ab+a^2b-a^3b) \\ e_4&=& 1/8(1-a+a^2-a^3-b+ab-a^2b+a^3b) \\
e_5&=& 1/2(1-a^2)\end{eqnarray*}

(\cite{seh} has  $-ab$ in $e_4$ which should be $+ab$ as above.)

The group ring matrices $\{E_1,E_2,E_3,E_4\}$ corresponding to
$\{e_1,e_2,e_3,e_4\}$ respectively have rank $1$ and the group ring matrix
$E_5$ corresponding to $e_5$ has
rank $4$, which can be seen from theory. Thus take the first columns of
$E_1,E_2,E_3,E_4$ and  4 linearly independent columns of $E_5$ to form
a matrix $P$. Then 
$P^{-1}AP=\diag(T_1,T_2,T_3,T_4,T_5)$ where $T_1,T_2,T_3,T_4$ are
scalars and $T_4$ is a $4\ti 4$ matrix, for any group ring matrix $A$ of
$K_8$.  

Precisely we may take: 

$P=\left(\begin{array}{c|c|c|c|cccc} 1&1&1&1&1&0&0&0\\  1&1&-1&-1
&0&1&0&0\\ 1&1&1&1&-1&0&0&0 \\ 1&1&-1&-1&0&-1&0&0\\  1&-1&1&-1
&0&0&1&0 \\ 1&-1&-1&1&0&0&0&1\\ 1&-1&1&-1 &0&0&-1&0 \\ 1&-1&-1&1
&0&0&0&-1 \end{array}\right)$

and then $P^{-1}AP=\diag(\al_1,\al_2,\al_3,\al_4,T)$ for any group
ring matrix $A$ of $K_8$ where $T$ is a $4\ti 4$ matrix. 

This gives an isomorphism from the group ring of $K_8$ to these block
matrices given by $w\mapsto \sigma(w)=W \mapsto P^{-1}WP$.
Representations of $K_8$ may be obtained by specialising to the group
elements. 

The following then gives an embedding of $K_8$: 

$a\mapsto \left(\begin{array}{cccc|cccc} 1 &0&0&0&0&0&0&0
  \\ 0&1&0&0&0&0&0&0 \\ 0&0&-1 
  & 0&0&0&0&0 \\ 0&0&0&-1 &0&0&0&0 \\ \hline 0&0&0&0&0&1&0&0
  \\ 0&0&0&0&-1&0&0&0 \\0&0&0&0&0&0&0&-1\\ 0&0&0&0&0&0 &1&
  0 \end{array}\right), 
b\mapsto \left(\begin{array}{cccc|cccc} 1 &0&0&0&0&0&0&0
  \\ 0&-1&0&0&0&0&0&0 \\ 0&0&1 
  & 0&0&0&0&0 \\ 0&0&0&-1 &0&0&0&0 \\ \hline 0&0&0&0&0&0&1&0
  \\ 0&0&0&0&0&0&0&1 \\0&0&0&0&-1&0&0&0\\ 0&0&0&0&0&-1 &0&
  0 \end{array}\right)$
 

Using the blocks gives other representations. For example

$a\mapsto \left(\begin{array}{cc|cc} 0&1&0&0
  \\ -1&0&0&0 \\ \hline 0&0&0&-1 \\ 0&0&1&0\end{array}\right), 
b\mapsto \left(\begin{array}{cc|cc} 0&0&1&0
  \\ 0&0&0&1 \\ \hline -1&0&0&0 \\ 0&-1&0&0 \end{array}\right)$
 
gives a representation of $K_8$. 

It may be shown directly from the
block form of $P$ that the
image of a group element has the $4\times 4$ block of the form 
$\begin{pmatrix}X& 0 \\ 0& Y \end{pmatrix}$ or else the form
$\begin{pmatrix} 0 & X \\ Y& 0\end{pmatrix}$ for $2\times 2$ blocks
  $X,Y$.  

\section{Abelian groups}\label{sec2} 
The abelian group case follows from the general case, Section
\ref{block}, but may be 
tackled directly and more illuminatingly as follows. 

Let $\{A_1,A_2, \ldots, A_k\}$ be an ordered set of matrices of the same size. 
Then the {\it block circulant matrix} formed from the set is 
$A=\cir(A_1,A_2,\ldots, A_k) = 
\begin{pmatrix} A_1 & A_2 & \ldots & A_k \\ A_k & A_1 & \ldots &
A_{k-1} \\ \vdots & \vdots & \vdots & \vdots \\ A_2 & A_3 & \ldots &
A_1 \end{pmatrix}$

If the $A_i$ have size $m\ti t$ then $A$ has size $km\ti kt$.
The block circulant formed depends on the order of the elements in  
 $\{A_1,A_2, \ldots, A_k\}$. 

Let $P$ be an $n\ti n$ matrix. Then the {\it block Fourier matrix} $P_f$
corresponding to $P$ is 
$P\otimes F$, the tensor product of $P$ and $F$ where 
$F$ is the Fourier $n\ti n$ matrix. 

Thus
$P_f= P\otimes F = \begin{pmatrix} P & P & P& \ldots & P \\ P &\om P & \om^2 P &
	       \ldots & \om^{n-1}P \\ \vdots &\vdots &\vdots & \vdots \\
	       P&\om^{n-1}P & \om^{2(n-1)} & \ldots & \om^{(n-1)(n-1)}P
	      \end{pmatrix}$

It is clear then that:
\begin{proposition}\label{invert} $P_f$ is invertible if and only if $P$ is invertible
 and the inverse when it exists is $P^{-1}\otimes F^*$.
\end{proposition}  

Here $F^*$ denotes the inverse of the Fourier matrix. If the Fourier
matrix is normalised in $\cc$, then $F^*$ is the complex conjugate
transposed of $F$.

The following theorem may be proved in a  manner similar to the proof
that the Fourier matrix diagonalises a circulant matrix.  
 \begin{theorem}\label{thm} Suppose $\{A_1,A_2, \ldots, A_k\}$ are matrices 
of the same size and  can
   be simultaneously diagonalised by $P$ with $P^{-1}A_iP = D_i$
   where each $D_i$ is diagonal. Then the block circulant matrix $A$ 
   formed from these matrices can be diagonalised by 
$P_f= P\otimes F= \begin{pmatrix} P&P&P&\ldots &P \\ P & \om P &\om^2P & \ldots &
     \om^{k-1}P \\ P & \om^2P &\om^{4}P & \ldots & \om^{2(k-1)}P
     \\ \vdots & \vdots &\vdots & \vdots &\vdots \\ P & \om^{k-1}P &
     \om^{2(k-1)}P & \ldots & \om^{(k-1)(k-1)}P \end{pmatrix}$ where
   $\om$ is a primitive $k^{th}$ root of unity. 

Moreover $P_f^{-1}AP_f = D$ where $D$ is diagonal and  

$D=\diag(D_1+D_2 + \ldots + D_k,
D_1+\om D_2 +\ldots + \om^{k-1}D_k, D_1+\om^2D_2+\om^4D_3 +
\ldots + \om^{2(k-1)}D_k, \ldots, D_1+\om^{k-1}D_2 + \om^{2(k-1)}D_3 +
\ldots + \om^{(k-1)(k-1)}D_k)$

\end{theorem}

\begin{proof}

The proof of this is direct, involving 
working out $AP_f$ and showing it is $P_fD$, with $D$ as given. 
Since $P_f$ is invertible by Proposition \ref{invert} the result will follow.
 This  is similar to a proof that the Fourier matrix
diagonalises a circulant matrix.  

\end{proof}

The simultaneous diagonalisation process of the Theorem may then be repeated.

Suppose now $G=K \cross H$, the direct product of $K,H$, and $H$ is
cyclic. Then a group
ring matrix of $G$ is of the form $M= \cir(K_1,K_2,\ldots, K_h)$ where
$K_i$ are group ring matrices of $K$ and $|H|=h$. If the $K_i$
can be diagonalised by $P$ then $M$ can be diagonalised by the Fourier
block matrix formed from $P$ by Theorem \ref{thm}.
A finite abelian  group is the direct product of
cyclic groups and thus repeating the process enables 
the simultaneous diagonalisation of the
group ring matrices of a finite abelian group and it gives an explicit
diagonalising matrix.  
The characters and character table of the finite abelian group may be read off 
from the diagonalising matrix. 

Since the Fourier $n\times n$ matrix diagonalises
a circulant $n\times n$ matrix, and the Fourier matrix is a
Hadamard complex matrix, the diagonalising matrix $P$ of size $q\times
q$, constructed by iteration of Theorem \ref{thm}, of a group ring
matrix of a finite abelian group is then seen to  satisfy $PP^*=qI$ 
and to have roots of unity as entries. It is thus a special type of
Hadmard complex matrix. 

The examples below illustrate the method.
 \subsection{Examples}
\begin{itemize}
\item Consider $G= C_3 \cross C_3$. Now $P=\begin{pmatrix} 1 & 1
& 1  \\ 1 & \om & \om^2 & \\ 1 & \om^2 & \om \end{pmatrix}$ where $\om$
is a primitive $3^{rd}$ root of unity diagonalises any circulant $3\ti
3$ matrix which is the group ring matrix of $C_3$. Then
$P_f= \begin{pmatrix} P&P&P \\ P & \om P & \om^2 P \\ P & \om^2 P & \om
  P \end{pmatrix}$ diagonalises any group ring matrix of $C_3\cross
C_3$.

Written out in full:  $P_f=\left(\begin{array}{ccc|ccc|ccc} 1&1 &1&1&1&1 &1&1&1 \\ 1
  & \om & \om^2 & 1 & \om &\om^2 & 1 & \om&\om^2 \\ 1 &\om^2 & \om & 1
  &\om^2 &\om & \ 1 &\om^2 & \om \\ \hline 1 & 1 &1 & \om &\om &\om &
  \om^2&\om^2&\om^2 \\ 1&\om &\om^2 & \om&\om^2 &1 &\om^2 &1 & \om
  \\ 1&\om^2&\om &\om &1&\om^2&\om^2&\om&1 \\ \hline 1&1&1& 
  \om^2&\om^2&\om^2&\om&\om&\om \\ 1&\om&\om^2
  &\om^2&1&\om &\om &\om^2&1 
  \\1&\om^2&\om&\om^2&\om&1&\om&1&\om^2 \end{array}\right)$

The characters and character table 
of $C_3\cross C_3$ may be read off from the rows of $P_f$ by labelling
the rows of $P_f$ appropriate to the listing of the elements of
$C_3\cross C_3$ when forming the group ring matrices. The listing here
is $\{1,g,g^2,h, hg,hg^2, h^2,h^2g,h^2g^2\}$ where the $C_3$ are
generated by $\{g,h\}$ respectively. Thus the character table of
$C_3\cross C_3$ is 

 $\left(\begin{array}{ccccccccc} 1&g&g^2 &h&hg&hg^2 & h^2& hg^2&
  h^2g^2 \\ \hline \hline  1&1 &1&1&1&1 &1&1&1 \\ 1
  & \om & \om^2 & 1 & \om &\om^2 & 1 & \om&\om^2 \\ 1 &\om^2 & \om & 1
  &\om^2 &\om & \ 1 &\om^2 & \om \\  1 & 1 &1 & \om &\om &\om &
  \om^2&\om^2&\om^2 \\ 1&\om &\om^2 & \om&\om^2 &1 &\om^2 &1 & \om
  \\ 1&\om^2&\om &\om &1&\om^2&\om^2&\om&1 \\  1&1&1& 
  \om^2&\om^2&\om^2&\om&\om&\om \\ 1&\om&\om^2
  &\om^2&1&\om &\om &\om^2&1 
  \\1&\om^2&\om&\om^2&\om&1&\om&1&\om^2 \end{array}\right)$

Note that $\frac{1}{\sqrt{9}}P_f$ is unitary and that
$P_f$ is a Hadamard complex matrix. 
\item For $C_2\cross C_4$ consider $P=\begin{pmatrix} 1&1
\\ 1&-1\end{pmatrix} $ and note that a
primitive $4^{th}$ root of 1 is $i=\sqrt{-1}$. Then $i^2=-1,
i^3=-i,i^4=1$.
  Now form 
$Q=\begin{pmatrix}  P&P&P&P \\ P&iP&-P&-iP \\ P&-P&P&-P
    \\ P&-iP&-P&iP \end{pmatrix}$. The characters of $C_2\cross C_4$
  can be read off from $Q$, $Q$ is a Hadamard complex matrix and
  $\frac{1}{\sqrt{8}}Q$ is unitary. 

\item 
For $C_3\cross C_4$ consider that $C_3\cross C_4
\cong C_{12}$. Then the diagonalising 
matrix obtained using the natural ordering in $C_3\cross C_4$ is
equivalent to the diagonalising matrix using the natural ordering in
$C_{12}$. 
 
\item 
Consider $C_2^n$. Let $P_1= \begin{pmatrix} 1&1 \\ 1& -1 \end{pmatrix}$
and inductively define for $n\geq 2$, $P_n = \begin{pmatrix}
						   P_{n-1} & P_{n-1} \\
						   P_{n-1} &
					      -P_{n-1}\end{pmatrix}$.

Then $P_n$ diagonalises any $\cc C_2^n$-matrix and the characters of $C_2^n$
 may be read off from the rows of $P_n$. Note that  $P_n$ is a Hadamard (real)
matrix.

\end{itemize}

\end{document}